\documentclass[8pt]{article}
\usepackage[utf8]{inputenc}
\usepackage[T2A]{fontenc}
\usepackage[english]{babel}

\usepackage{amsfonts,amssymb,mathrsfs,amscd,amsmath,amsthm}
\usepackage[pdftex]{graphicx}
\usepackage{verbatim}
\usepackage{misccorr}
\usepackage{mathtools}
\usepackage{indentfirst}
\usepackage{alltt}
\usepackage{stmaryrd}
\usepackage{enumitem}
\usepackage{xurl}

\usepackage[a5paper,margin=1cm,bmargin=1.5cm]{geometry}
\usepackage{parskip}

\usepackage{amsmath}
\usepackage{amssymb}
\usepackage{amsthm}
\usepackage{amsfonts}
\usepackage{asymptote}
\usepackage{hyperref}
\usepackage{xcolor}
\usepackage{titling}
\usepackage{lipsum}

\newtheorem{lemma}{Lemma}[]
\newtheorem{claim}{Claim}[]

\newtheorem{theorem}{Theorem}[]
\newtheorem{corollary}{Corollary}[]

\newcommand{\diam}{\operatorname{diam}}

\newcommand{\dist}{\operatorname{d}}

\newcommand{\ed}{\operatorname{ed}}

\newcommand{\dis}{\operatorname{dis}}

\title{ Gromov--Hausdorff distances between normed spaces}

\author{I.\,N. Mikhailov}

\date{}

\begin{document}
\maketitle

\begin{abstract}
In the present paper we study the original Gromov--Hausdorff distance between real normed spaces. In the first part of the paper we prove that two finite-dimensional real normed spaces on a finite Gromov--Hausdorff distance are isometric to each other.  We then study the properties of finite point sets in finite-dimensional normed spaces whose cardinalities exceed the equilateral dimension of an ambient space. By means of the obtained results we prove the following enhancement of the aforementioned theorem: every finite-dimensional normed space lies on an infinite Gromov--Hausdorff distance from all other non-isometric normed spaces.

\end{abstract}
\begin{center}
    \section{Introduction}
\end{center}

The Gromov--Hausdorff distance is one of the most beautiful constructions in metric geometry. It allows to compare how similar two arbitrary metric spaces are. The concept was first introduced in \cite{13} by D.\,Edwards. Later it became famous due to M.\,Gromov's paper \cite{14} (see historical details in \cite{15}). The classical Gromov--Hausdorff distance between metric spaces $X$ and $Y$ is defined as the infinum of the Hausdorff distances between the images $X'$ and $Y'$ of the spaces $X$ и $Y$ under all possible isometric embeddings $\phi\colon X\to Z$ и $\psi\colon Y\to Z$ into an arbitrary metric space $Z$.  

In the majority of applications the Gromov--Hausdorff distance is used to study compact metric spaces. In case of unbounded metric spaces the punctured spaces are considered, i.e., the pairs $(X,\,p)$ where $X$ is an arbitrary metric space and $p$ is one of its points. Traditionally, in such situation the topology generated by the Gromov--Hausdorff distance is of more interest than the distance itself. By $B_r(x)$ let us denote a closed ball of radius $r$ centered at the point $x$. Then the sequence of punctured metric spaces $(X_n,\, p_n)$ converges by Gromov--Hausdorff to a punctured metric space  $(Z, \,p)$, iff the following conditions hold (see details in \cite{1}, \cite{16}): for every $r > 0$ and $\varepsilon > 0$ there exists a natural number $n_0$ such that for every $n > n_0$ there exists a mapping $f\colon B_r(p_n)\to X$ with properties
\begin{enumerate} 
\item $f(p_n) = p$;
\item $\dis(f):= \sup\Bigl\{\bigl||xx'|-|f(x)f(x')|\bigr|\colon x,\,x'\in X\Bigr\} < \varepsilon$;
\item $\varepsilon$-neighbourhood of the set $f(B_r(p_n))$ in $X$ contains the ball $B_{r-\varepsilon}(p)$.
\end{enumerate}

In the present paper we study the Gromov--Hausdorff distance in its original sense. Restrictions of the Gromov--Hausdorff distance to various classes of metric spaces possess fascinating and often unexpected geometric properties. For example, let us consider the class (in the sense of NBG set theory) $\mathcal{GH}$ of representatives of all isometry classes of metric spaces. The Gromov--Hausdorff distance defines a generalised pseudometric on this class, i.e., is a symmetric, non-zero function, that satisfies the triangle inequality. This space is called \emph{the Gromov-Hausdorff class}. The equivalence classes in $\mathcal{GH}$ under the relation: $X\sim Y$ iff $\dist_{GH}(X,\,Y)<\infty$ --- are called \emph{clouds}. In the monography \cite{17} M. Gromov announced that each cloud in $\mathcal{GH}$ is contractible and gave an example of the cloud corresponding to $\mathbb{R}^n$. However, it turned out later that there exist clouds which are not even invariant under the multiplication of all their spaces by some positive number $\lambda$ (see details in paper \cite{18}).  

In the paper we are studying the Gromov--Hausdorff distance between real normed spaces. Unlike in the paper \cite{19}, we consider the global Gromov--Hausdorff distance rather than the Gromov--Hausdorff distance between the closed unit balls of normed spaces. For an arbitrary $\varepsilon > 0$ let us call a linear mapping $T$ between normed spaces $E$ and $F$ an $\varepsilon$-isometry, if for all $x,\,y\in E$ an inequality holds $\Bigl|\|x-y\|-\|T(x)-T(y)\|\Bigr|\leqslant\varepsilon$. In the paper \cite{20} the following theorem is proved: if $E$ and $F$ are finite-dimensional normed spaces and $T\colon E\to F$ is a surjective $\varepsilon$-isometry for some $\varepsilon$, then there exists an isometry $I\colon E\to F$ such that $\|T(x)-I(x)\|\leqslant 5\varepsilon$ for all $x\in E$. It follows from this result that any two finite-dimensional normed spaces on a finite Gromov--Hausdorff distance are isometric. In section $3$ we give a simpler proof of this theorem, based on Theorem \ref{theorem: base-lemma}, which provides a natural sufficient condition for isometric embeddability of a given bounded metric space into a given finite-dimensional normed space.  

\emph{Equilateral dimension} of a metric space is the largest cardinality of such its subset that all the distances between different points of this subset are pairwise equal to each other. In the paper \cite{12} it is proved that the equilateral dimension of an arbitrary $n$-dimensional normed space does not exceed $2^n$. In section $4$ we study some properties of finite sets of points in normed spaces whose cardinalities exceed the equilateral dimension of the ambient space. Based on the obtained results, we prove a new lower estimate on the Gromov--Hausdorff distance between a finite-dimensional normed space and an arbitrary metric space of a larger equilateral dimension (see Theorem \ref{theorem: second_ineq}). From this estimate we obtain the following enhancement of the main theorem of section 3: every finite-dimensional normed space lies on an infinite Gromov--Hausdorff distance from all other non-isometric normed spaces.

The author expresses deep gratitude to his scientific advisor professor A.\,A.\,Tuzhilin and also professor
A.\,O.\,Ivanov for posing the problem and constant attention to the work.

\begin{center}
\section{Main definitions and preliminary results}
\end{center}

We start with the introduction of some basic notation. Let $X$ be an arbitrary metric space. We denote the distance between the points $x$ and $y$ of $X$ by $|xy|$. Let $U_r(a) =\{x\in X\colon |ax|<r\}$, $B_r(a) = \{x\in X\colon |ax|\leqslant r\}$ be an open and a closed ball of radius $r$ centered at the point $a$ correspondingly. For an arbitrary subset $A\subset X$ we define $U_r(A) = \cup_{a\in A} U_r(a)$ --- an open $r$-neighbourhood of $A$. For non-empty subsets $A\subset X$ and $B\subset X$ by $\dist(A,\,B)$ we denote a simple distance between these subsets, namely, $\dist(A,\,B) = \inf\bigl\{|ab|\colon\,a\in A,\,b\in B\bigl\}$. 

{\bf Definition 1.} Let $A$ and $B$ be non-empty subsets of a metric space $X$. \textit{The Hausdorff distance} between $A$ and $B$ is the value $$\dist_H(A,\,B) = \inf\{r > 0\colon A\subset U_r(B),\,B\subset U_r(A)\}.$$

{\bf Definition 2.} Let $X$ and $Y$ be metric spaces. If $X',\,Y'$ are subsets of a metric space $Z$ such that $X'$ is isometric to $X$ and $Y'$ is isometric to $Y$, then we call the triple $(X',\,Y',\,Z)$ \textit{a metric realization of a pair $(X,\,Y)$}.

{\bf Definition 3.} \textit{The Gromov--Hausdorff distance} $\dist_{GH}(X,\,Y)$ between two metric spaces $X$, $Y$ is the infinum of positive numbers $r$ such that there exists a metric realization $(X',\,Y',\,Z)$ of a pair $(X,\,Y)$ with $\dist_H(X',\,Y') \leqslant r$.

Let now $X,\,Y$ be non-empty sets. 

{\bf Definition 4.} Any subset $\sigma\subset X\times Y$ is called a \textit{relation} between $X$ and $Y$.

Denote the set of all non-empty relations between $X$ and $Y$ by $\mathcal{P}_0(X,\,Y)$.

Define $$\pi_X\colon X\times Y\rightarrow X,\;\pi_X(x,\,y) = x,$$ $$\pi_Y\colon X\times Y\rightarrow Y,\;\pi_Y(x,\,y) = y.$$ 

{\bf Definition 5.} Relation $R\subset X\times Y$ is called a \textit{correspondence}, if $\pi_X|_R$ and $\pi_Y|_R$ are surjective.

Denote the set of all correspondences between $X$ and $Y$ by $\mathcal{R}(X,\,Y)$.

{\bf Definition 6.} Let $X,\,Y$ be arbitrary metric spaces. Then for every $\sigma \in \mathcal{P}_0(X,\,Y)$ \textit{the distortion} of $\sigma$ is defined as $$\dis \sigma = \sup\Bigl\{\bigl||xx'|-|yy'|\bigr|\colon(x,\,y),\,(x',\,y')\in\sigma\Bigr\}.$$

\begin{claim}[\cite{1}] \label{claim: distGHformula}
For arbitrary metric spaces $X$ and $Y$ the following equality holds $$2\dist_{GH}(X,\,Y) = \inf\bigl\{\dis\,R\colon R\in\mathcal{R}(X,\,Y)\bigr\}.$$
\end{claim}

{\bf Definition 7.} Let $X$ be an arbitrary metric space. By $\mathcal{H}(X)$ we denote the set of all non-empty closed bounded subsets in $X$. The Hausdorff distance defines a finite metric on $\mathcal{H}(X)$ (see, for example, \cite{1}). The resulting metric space is called a \textit{hyperspace} of the space $X$.

\begin{theorem}[\cite{1}]\label{theorem: theorem1}
Let $X$ be an arbitrary metric space. Then $X$ is compact iff $\mathcal{H}(X)$ is compact.
\end{theorem}

\begin{theorem}[\cite{1}]\label{theorem: theorem2}
If $X$ is a compact metric space and $Y$ is a complete metric space such that $\dist_{GH}(X,\,Y) = 0$, then $X$ is isometric to $Y$.
\end{theorem}

{\bf Definition 8.} The subset of an arbitrary metric space $X$ is called \textit{equilateral} if all the distances between its distinct points are pairwise equal to each other. The \textit{equilateral dimension} of a metric space $X$ is the largest cardinality of its equilateral subset. We denote the equilateral dimension of $X$ by $\ed(X)$. 

\begin{theorem}[\cite{12}]\label{theorem: Soltan}
Let $V$ be an $n$-dimensional normed space. Then $\ed(V)\leqslant 2^n$.
\end{theorem}

\begin{theorem}[\cite{11}]\label{theorem: infEd}
Let $X$ be an arbitrary infinite-dimensional normed space. Then for an arbitrary $m\in\mathbb{N}$ there exists a finite equilateral subset in $X$ of $m$ points.
\end{theorem}

{\bf Definition 9.} Let $X$ be an arbitrary metric space, and $\varepsilon > 0$ an arbitrary positive number. The subset $P\subset X$ is called $\varepsilon$-\emph{separated}, if for any two points $p$ and $q$ in $P$ the inequality holds $|pq|\geqslant\varepsilon$. 

{\bf Definition 10.} Let $V$ be a normed space, and $m$ be some positive integer. We denote the infinum of positive numbers $r$ such that there exists a $1$-separated set of $m$ points in $B_r(0)$ by $R_m(V)$.

Finally, we need a few classical results from analysis.

\begin{theorem}[\cite{1}, \cite{10}]\label{theorem: theorem3}
Let $K$ be a compact metric space. Then an arbitrary isometric mapping $f\colon K\to K$ is surjective and, hence, is an isometry.
\end{theorem}

\begin{theorem}[\cite{9}]\label{theorem: theorem6}
Two arbitrary normed spaces are isometric iff their closed unit balls are isometric.
\end{theorem}

\begin{center}
\section{Finite-dimensional case}
\end{center}

In this section we show that any two finite-dimensional normed spaces on a finite Gromov--Hausdorff distance are isometric to each other.

\begin{theorem}\label{theorem: base-lemma}
Let $V$ be a finite-dimensional normed space and $X$ be an arbitrary bounded metric space. Suppose there exists a sequence of positive numbers $(\varepsilon_n)_{n\in\mathbb{N}}$ converging to $0$ such that for any $n$ there exists a mapping $f_n\colon X\to V$ with $\dis(f_n) \leqslant \varepsilon_n$. Then the completion $\tilde{X}$ of $X$ is compact and both $X$ and $\tilde{X}$ can be isometrically embedded into $V$.
\end{theorem}

\begin{proof}
Fix an arbitrary $x\in X$. By $t_n$ denote a translation of $V$ by the vector $-f_n(x)$. Let $g_n =t_n \circ f_n$. Since $t_n$ is an isometry, we have $\dis\,g_n = \dis\,f_n$. For every $n\in\mathbb{N}$ define $X_n = \overline{g_n(X)}$ --- a closure of $g_n(X)$ in the topology on $V$ generated by the norm. By the definition of the distortion, for any points $a,\,b\in X$ the inequality holds $\bigl| \|a-b\| - \|g_n(a) - g_n(b)\|\bigr|\leqslant \varepsilon_n$. According to Claim \ref{claim: distGHformula}, we obtain that $\dist_{GH}\bigl(X,\,g_n(X)\bigr) \leqslant \frac{\varepsilon_n}{2}$. Since $\dist_{GH}\bigl(g_n(X),\,\overline{g_n(X)}\bigr) = 0$, by applying the triangle inequality we obtain that $$\dist_{GH}(X,\,X_n) \leqslant \dist_{GH}\bigl(X,\,g_n(X)\bigr) + \dist_{GH}(g_n(X),\,X_n) \leqslant \frac{\varepsilon_n}{2}.$$ Note that for every $n\in\mathbb{N}$ the point $0$ belongs to the set $X_n$. Also $\diam\,X_n\leqslant \diam\,X+\varepsilon_n$. Hence, all $X_n$ belong to the ball $B := B_{\diam(X)+c}(0)$, where $c$ is some constant such that $\varepsilon_n\leqslant c$ for every $n$. The space $V$ is finite-dimensional, thus, $B$ is compact. Therefore, $\mathcal{H}(B)$ is compact by Theorem \ref{theorem: theorem1}. Sets $X_n$ are closed and bounded, so they form a sequence of points in $\mathcal{H}(B)$. 
Then it follows from the compactness of $\mathcal{H}(B)$ that there exists a subsequence $X_{n_s}$ that converges to some non-empty closed and bounded subset $A$ of $B$. Note that $A$ is compact as the closed subset of a compact ball $B$. Then for every $s$ the following inequalities hold $$\dist_{GH}(X,\,A)\leqslant \dist_{GH}(X,\,X_{n_s}) + \dist_{GH}(X_{n_s},\,A) \leqslant \dist_{GH}(X,\,X_{n_s})+\dist_{H}(X_{n_s},\,A).$$ The right side tends to $0$ when $s\to\infty$. Thus, $\dist_{GH}(X,\,A) = 0$.  Since $\dist_{GH}(X,\,\tilde{X}) = 0$, the triangle inequality implies, that $\dist_{GH}(\tilde{X},\,A) \leqslant \dist_{GH}(\tilde{X},\,X) + \dist_{GH}(X,\,A) = 0$, i.e., $\dist_{GH}(\tilde{X},\,A) = 0$. Then from Theorem \ref{theorem: theorem2} it follows that $\tilde{X}$ and $A$ are isometric. It implies that $\tilde{X}$ is compact and can be isometrically embedded into $V$. In particular, $X$ can be isometrically embedded into $V$.  
\end{proof}

\begin{lemma} \label{lemma: lessthaninf}
If $X,\,Y$ are normed spaces such that $\dist_{GH}(X,\,Y) < \infty$, then $\dist_{GH}(X,\,Y) = 0$.
\end{lemma}

\begin{proof}
Suppose $\dist_{GH}(X,\,Y) = c < \infty$. Then for an arbitrary $\lambda > 0$ the following equalities hold $c = \dist_{GH}(X,\,Y) = \dist_{GH}(\lambda X,\,\lambda Y) = \lambda \dist_{GH}(X,\,Y) = \lambda c$. Hence, $c = 0$.
\end{proof}

\begin{theorem} \label{theorem: base-iso}
If $V,\,W$ are arbitrary finite-dimensional normed spaces with $\dist_{GH}(V,\,W) < \infty$, then $V$ and $W$ are isometric.
\end{theorem}

\begin{proof}
From Lemma \ref{lemma: lessthaninf} we have $\dist_{GH}(V,\,W) = 0$. Then, by Claim \ref{claim: distGHformula}, for an arbitrary positive number $\varepsilon > 0$ there exists a correspondence $R$ between $V$ and $W$ with the distortion $\dis\,R \leqslant \varepsilon$.

Let us consider the unit balls $B_1 = B^V_1(0),\,B_2 = B^W_1(0)$ of the spaces $V,\,W$ correspondingly. It follows from Theorem \ref{theorem: base-lemma} that there exist isometric maps $f_1\colon B_1\to B_2$ и $f_2\colon B_2\to B_1$.
Then $f_2\circ f_1\colon B_1\to B_1$ and $f_1\circ f_2\colon B_2\to B_2$ are isometric embeddings. From Theorem \ref{theorem: theorem3} we conclude that the constructed mappings $f_2\circ f_1$, $f_1\circ f_2$ are isometries. It follows immediately that $f_1$ and $f_2$ are bijective, so the balls $B_1$ and $B_2$ are isometric to each other. From Theorem \ref{theorem: theorem6} we conclude that the spaces $X$ and $Y$ are isometric. 
\end{proof}

{\bf Example 1.} Let us consider any two non-isometric normed spaces $X$ and $Y$. For example, two copies of $\mathbb{R}^2$ one with the euclidean norm and one with the $\max$-norm. Let us also consider any $\varepsilon_1$-network $\sigma_1$ in $X$ and any $\varepsilon_2$-network $\sigma_2$ in $Y$. Then it immediately follows from the proven theorem that $\dist_{GH}(\sigma_1,\,\sigma_2)=\infty$. In particular, the Gromov--Hausdorff distance between $\mathbb{Z}^2$ with metrics induced from $X$ and $Y$ respectively is infinite.

\begin{center}
\section{Metric imbalance}
\end{center}

We start this section with a few results about point sets in finite-dimensional normed spaces whose cardinalities exceed the equilateral dimension of the corresponding space (which is finite due to Theorem \ref{theorem: Soltan}). Firstly, we prove an important inequality that characterises such point sets.

\begin{theorem}\label{theorem: first_ineq_normed}
Let $V$ be a finite-dimensional normed space with $\ed(V) = p$. Then there exists a constant $c>0$ such that for an arbitrary set of $m > p$ distinct points $v_1,\,\ldots,\,v_m$ in $V$ there exist three of them $\{v_i,\,v_j,\,v_k\}$ such that
$$\varphi(v_i,\,v_j,\,v_k) := \left|\frac{\|v_i-v_k\|}{\|v_j - v_k\|}  - 1\right| \geqslant c.$$ 
\end{theorem}

\begin{proof}
Arguing by contradiction, suppose that for every $\varepsilon_n = \frac{1}{n},\,n\in\mathbb{N},\,n\geqslant 6$ there exist distinct points $v_1^n,\,\ldots,\,v_m^n$ such that for all different $i,\,j,\,k$ the following inequalities hold $\varphi(v^n_i,\,v^n_j,\,v^n_k)<\varepsilon_n$. For every $\lambda > 0$ we have $\varphi(\lambda v^n_i,\,\lambda v^n_j,\,\lambda v^n_k) = \varphi(v^n_i,\,v^n_j,\,v^n_k)$, so without loss of generality we assume that $\|v^n_1-v^n_2\|=1$. Then inequalities $\varphi(v^n_i,\,v^n_2,\,v^n_1)<\varepsilon_n$ imply that $\bigl|\|v^n_1v^n_i\|-1\bigr|\leqslant\varepsilon_n$ for each $i=1,\,\ldots,\,m$. Hence, from $\varphi(v^n_j,\,v^n_1,\,v^n_i)<\varepsilon_n$ it now follows that $$\bigl|\|v^n_i-v^n_j\|-\|v^n_1-v^n_i\|\bigr|\leqslant\varepsilon_n\|v_1-v_i\|\leqslant \varepsilon_n(1+\varepsilon_n)< 2\varepsilon_n.$$ By triangle inequality, $$\bigl|\|v^n_i-v^n_j\|-1\bigr|\leqslant \bigl|\|v^n_i-v^n_j\|-\|v^n_1-v^n_i\|\bigr| + \bigl|\|v^n_1-v^n_i\|-1\bigr| < 3\varepsilon_n.$$ Thus, for every distinct $i$ and $j$ the following equality holds $\|v^n_i-v^n_j\| = 1+\delta_n$ where $|\delta_n|<3\varepsilon_n$. In particular, $\|v^n_i-v^n_j\|\geqslant \frac{1}{2}$ for $i\neq j$.

Let us put $\xi^n = (0,\,v_2^n-v_1^n\ldots,\,v_m^n-v_1^n)\in V^{m}$ and consider $V^m$ with the norm $$\|(v_1,\,v_2,\,\ldots,\,v_m)\| = \max_q \|v_q\|.$$ 

Note that the sequence $\xi^n$ is bounded due to already established inequalities $$\bigl|\|v^n_i - v^n_1\|-1\bigr|\leqslant \varepsilon_n,\; i=1,\,\ldots,\,m.$$ Consider the subset $P = \bigl\{(v_1,\,\ldots,\,v_m)\colon \|v_i-v_j\| \geqslant \frac{1}{2} \;\forall\,i\neq j\bigr\}\subset V^m$ with the metric induced from $V^m$. Since $\|v^n_i-v^n_j\|\geqslant \frac{1}{2}$ for $i\neq j$, the sequence $\xi^n$ lies in $P$. The space $V^m$ is finite-dimensional so there exists a subsequence $\xi^{n_s}$ that converges to some $\xi = (w_1,\,\ldots,\,w_m)$, which belongs to $P$ due to the closeness of $P$ in $V^m$. Define a function $\tilde{\varphi}\colon V^m\to \mathbb{R}$ by the formula $$\tilde{\varphi}(v_1,\,\ldots,\,v_m) = \max_{i\neq j\neq k\neq i} \varphi(v_i,\,v_j,\,v_k).$$ Note that $\tilde{\varphi}$ is continuous on $P$. Since $\varphi(v^n_i,\,v^n_j,\,v^n_k) = \varphi(v^n_i-v^n_1,\,v^n_j-v^n_1,\,v^n_k-v^n_1)$, it follows from the inequalities $\varphi(v^n_i,\,v^n_j,\,v^n_k)<\varepsilon_n$ that $\varphi(v^n_i-v^n_1,\,v^n_j-v^n_1,\,v^n_k-v^n_1)<\varepsilon_n$. The latter inequalities are equivalent to $\tilde{\varphi}(0,\,v^n_2-v^n_1,\,\ldots,\,v^n_m-v^n_1) < \varepsilon_n$. Hence, from the continuity of $\tilde{\varphi}$ on $P$ we conclude that $\tilde{\varphi}(w_1,\,\ldots,\,w_m) = 0$. Thus, for every different $i,\,j,\,k$ we have $\varphi(w_i,\,w_j,\,w_k) = 0$ which is equivalent to $\|w_i-w_k\| = \|w_j-w_k\|$. Therefore, $(w_1,\,\ldots,\,w_m)$ is an equilateral set in $V$ of cardinality $m > \ed(V) = p$ --- contradiction.
\end{proof}

Let $V$ be an arbitrary finite-dimensional normed space. Let us also fix a positive integer $m$. We define a \emph{metric imbalance} of $V$ of order $m$ as follows $$c_m(V) = \sup\left\{ c\colon \forall\,(v_1,\,\ldots,\,v_m)\in V^m\,\exists\,i\neq j,\,j\neq k,\,k\neq i\colon\left|\frac{|v_iv_k|}{|v_jv_k|}-1\right|\geqslant c\right\}.$$
We conclude immediately from this definition that metric imbalance is a non-decreasing function of $m$. Besides, Theorem \ref{theorem: first_ineq_normed} implies that if $m > \ed(V)$ then $c_m(V) > 0$.  

\begin{theorem}\label{theorem: nicequantity}
Let $V$ be a finite-dimensional normed space. Denote its metric imbalance by $c_m:=c_m(V)$ and 
$R_m:=R_m(V)$. Then the following inequalities hold $$2R_m + 1 \geqslant c_m \geqslant R_m-2.$$
\end{theorem}

\begin{proof}
Let us put $t_n = c_m+\frac{1}{n}$, $n \in \mathbb{N}$. By definition of $c_m$ there exists such set of distinct points $(x^n_1,\,\ldots,\,x^n_m)$ that for all different $i,\,j,\,k$ the following inequalities hold $\varphi(x^n_i,\,x^n_j,\,x^n_k) \leqslant t_n$. 

Note that for every $\lambda > 0$ it holds $\varphi(x^n_i,\,x^n_j,\,x^n_k) = \varphi(\lambda x^n_i,\,\lambda 
 x^n_j,\,\lambda  x^n_k)$. Hence, by putting $\lambda = \min_{i\neq j}|x^n_ix^n_j|$ and $y^n_i = \frac{x^n_i}{\lambda}$, we obtain such set of points $(y^n_1,\,\ldots,\,y^n_m)$ that for all different $i,\,j,\,k$ we have $\varphi(y^n_i,\,y^n_j,\,y^n_k)\leqslant t_n$, and also $\min_{i\neq j}|y^n_iy^n_j|=1$. By renumerating the points, without loss of generality we can assume that $\min_{i\neq j}|y^n_iy^n_j| = |y^n_1y^n_2|$. Let us put $z^n_j = y^n_j-y^n_1$. Since $\varphi(y^n_i,\,y^n_j,\,y^n_k) = \varphi(y^n_i-y^n_1,\, 
 y^n_j-y^n_1,\,y^n_k-y^n_1) = \varphi(z^n_i,\,z^n_j,\,z^n_k)$, a set of points $(z^n_1,\,\ldots,\,z^n_m)$ has the following properties: for every different $i,\,j,\,k$ the inequalities hold $\varphi(z^n_i,\,z^n_j,\,z^n_k)\leqslant t_n$; $\min_{i\neq j}|z^n_iz^n_j| = |z^n_1z^n_2| = 1$, and, finally, $z^n_1 = 0$.
 
For each $k$ distinct from $1$ and $2$ the inequality holds $\varphi(z^n_k,\,z^n_2,\,z^n_1)\leqslant t_n$, therefore, $\left|\frac{|z^n_kz^n_1|}{|z^n_2z^n_1|}-1\right| = \Bigl||z^n_1z^n_k|-1\Bigr|\leqslant t_n$ from which it follows that $|z^n_1z^n_k|\leqslant t_n+1$. Hence, for every $n$ all the points $z^n_k$ belong to the ball $B_{t_n+1}(0)$. Besides, $\min_{i\neq j}|z^n_iz^n_j| = 1$. Then by definition of $R_m$ for every $n$ the inequality holds $R_m \leqslant c_m+1+\frac{1}{n}$. Therefore, $R_m\leqslant c_m+1$.

Let us prove the second inequality. Consider an arbitrary $1$-separated set $a_1,\,\ldots,\,a_m$ in $V$. By definition of the metric imbalance there exist such different indices $i,\,j,\,k$ that $\Bigl|\frac{|a_ia_k|}{|a_ja_k|}-1\Bigr|\geqslant c_m$. Then $$|a_ia_k|\geqslant (c_m-1)|a_ja_k| \geqslant c_m-1,$$
which implies that it is impossible for a $1$-separated set of $m$ points to lie in a ball of a radius less than $\frac{c_m-1}{2}$. Hence, $R_m \geqslant \frac{c_m-1}{2}$ which is equivalent to $2R_m +1\geqslant c_m$.

\end{proof}

\begin{corollary}\label{theorem: bigsetsinfinitebehaviour}
If $V$ is a finite-dimensional normed space then $\lim_{m\to\infty} c_m(V) = \infty$.
\end{corollary}

\begin{proof}
Inequalities \ref{theorem: nicequantity} imply that it suffices to prove that $\lim_{m\to\infty}R_m = \infty$. Suppose that is not true. From the definition of $R_m$ we conclude that it is a non-decreasing function of $m$. Then it follows from our assumption that there exists a constant $C$ such that an inequality $C\geqslant R_m$ holds for all $m$. But this inequality would mean that it is possible to find an arbitrarily large $1$-separated set in $B_C(0)$ which contradicts its compactness.
\end{proof}

\begin{theorem} \label{theorem: second_ineq}
Let $Y$ be a finite-dimensional normed space with $\ed(Y) = n$. Also let $X$ be an arbitrary metric space with an equilateral subset $\{x_1,\,\ldots,\,x_m\}$ of diameter $d$, where $m > n$. Denote $c_m := c_m(Y)$. Then $$\dist_{GH}(X,\,Y) \geqslant \frac{1}{2}\min\Bigl\{\frac{d}{2},\,\frac{dc_m}{2+c_m}\Big\} > 0.$$
\end{theorem}

\begin{proof}
Let $R$ be an arbitrary correspondence between $X,\,Y$ with the distortion $\dis\,R = t$. At first, we consider the case $t < \frac{d}{2}$. Choose an arbitrary $y_i \in R(x_i)$. In the considered case all $y_i$ are automatically different. By the definition of distortion $\Bigl||y_iy_j| - |x_ix_j|\Bigr|\leqslant t$. Hence, the following inequalities hold $$\frac{-2t}{d+t} = \frac{|x_ix_k|-t}{|x_jx_k|+t} - 1 \leqslant \frac{|y_iy_k|}{|y_jy_k|} - 1 \leqslant \frac{|x_ix_k|+t}{|x_jx_k|-t} - 1 = \frac{2t}{d - t}.$$ According to Theorem \ref{theorem: first_ineq_normed} there exist such indices $i,\,j,\,k$ that $\displaystyle \left|\frac{|y_iy_k|}{|y_jy_k|}-1\right| \geqslant c_m$. Then $\displaystyle \frac{2t}{d - t} \geqslant c_m$, which implies $\displaystyle t\geqslant \frac{dc_m}{2+c_m}$. Thus, for an arbitrary correspondence $R\in \mathcal{R}(X,\,Y)$ its distortion is either not less than $\frac{d}{2}$, or not less than $\frac{dc_m}{2+c_m}$. It follows that $\dis\,R\geqslant \min\Bigl\{\frac{d}{2},\,\frac{dc_m}{2+c_m}\Bigr\}$. Hence, by Claim \ref{claim: distGHformula} we obtain that $\dist_{GH}(X,\,Y)\geqslant \frac{1}{2}\min\Bigl\{\frac{d}{2},\,\frac{dc_m}{2+c_m}\Big\}$. 
\end{proof}

Note that Corollary \ref{theorem: bigsetsinfinitebehaviour} implies that $\lim\limits_{m\to\infty} \frac{dc_m}{2+c_m} = d$. So for sufficiently large $m$ the estimate from Theorem \ref{theorem: second_ineq} turns into $\dist_{GH}(X,\,Y)\geqslant\frac{d}{4}.$

\begin{theorem}\label{theorem: general}
Let $X,\,Y$ be normed spaces such as $\dim(Y)<\infty$ and $\dist_{GH}(X,\,Y) < \infty$. Then $X$ is isometric to $Y$.
\end{theorem}

\begin{proof}
According to Theorem \ref{theorem: base-iso}, it suffices to prove that $X$ is finite-dimensional. Suppose $\dim(X) = \infty$. Theorem \ref{theorem: Soltan} implies that the equilateral dimension of $Y$ is finite. Denote $\ed(Y) = n$. By Theorem \ref{theorem: infEd} there exists an equilateral set $x_1,\,\ldots,\,x_m\in X$ of $m>n$ points. Denote its diameter by $d$. Then for an arbitrary $\lambda > 0$ points $\lambda x_1,\,\ldots,\,\lambda x_m$ form an equilateral set in $X$ of diameter $\lambda d$. Hence, there are equilateral sets of $m > n$ points and sufficiently large diameters in $X$. Now Theorem  \ref{theorem: second_ineq} implies that $\dist_{GH}(X,\,Y) = \infty$ --- contradiction. 
\end{proof}

% Lomonosov Moscow State University, Faculty of Mechanics and Mathematics,\\ Leninskie Gory 1, Moscow, 119991, Russia\\
% \emph{E-mail:} ivan.mikhailov@math.msu.ru.


\begin{thebibliography}{2}
\bibitem{18} S.\,Bogatyy,\,A.\,A.\,Tuzhilin, \emph{Actions of similarity transformation on families of metric spaces} [In Russian], Itogi Nauki i Tekhniki, Seriya Sovremennaya Matematika i Ee Prilozheniya Tematicheskie Obzory, {\bf 223} (2023), 3-13.
\bibitem{11} P.\,Brass, \emph{On equilateral simplices in normed spaces}, Beiträge Algebra Geom. \textbf{40} (1999), pp. 303-307.
\bibitem{1} D.\,Yu.\,Burago, Yu.\,D.\,Burago, S.\,V.\,Ivanov, \emph{A course in metric geometry}, Moscow-Izhevsk, Institute for Computer Research, 2004.
\bibitem{13} D.\,Edwards, \emph{The structure of superspace}, Studies in Topology, Academic Press, 1975.
\bibitem{14} M.\,Gromov, \emph{Groups of polynomial growth and expanding maps}, Publications Mathematiques I.H.E.S., \textbf{53} 1981. 
\bibitem{17} M.\,Gromov, \emph{Structures metriques pour les varietes riemanniennes}, Textes mathematiques. Recherche (v. 1), CEDIC/Fernand Nathan, 1981.
\bibitem{20}  P.\,M.\,Gruber, \emph{Stability of isometries}, Trans. Amer. Math. Soc. \textbf{245} (1978), 263-277.
\bibitem{16} D.\,A.\,Herron, \emph{Gromov-Hausdorff Distance for Pointed Metric Spaces}, J. Anal., \textbf{24(1)} (2016), pp 1–38.
\bibitem{19} N.\,J.\,Kalton, M.\,I.\,Ostrovskii, \emph{Distances between Banach spaces}, Forum Math. \textbf{11} (1999), 17-48.
\bibitem{9} P.\,Mankiewicz, \emph{On extension of isometries in normed linear spaces}, Bull. Acad. Polon. Sci. Ser. Sci. Math. Astronomy, Phys. \textbf{20} (1972), 367–371.
\bibitem{12} S.\,P.\,Soltan, \emph{Analogues of regular symplexes in normed spaces}, \emph{Докл. АН СССР}, \textbf{222(6)} (1975), 1303-1305.
\bibitem{15} A.\,A.\,Tuzhilin, \emph{Who invented the Gromov-Hausdorff Distance?}, 2016, ArXiv e-prints, arXiv:1612.00728.
\bibitem{10} R.\,Wobst, \emph{Isometrien in metrischen Vektorräumen}, Studia Math. \textbf{54} (1975), pp. 41-54.
\end{thebibliography}
\end{document}